\documentclass[a4paper,12pt]{article}

\usepackage{epsfig}
\usepackage{amssymb,amsmath,amsthm,color}

\newtheorem{proposition}{Proposition}
\newtheorem{theorem}{Theorem}
\newtheorem{lemma}{Lemma}
\newtheorem{remark}{Remark}
\definecolor{darkred}{rgb}{0.9,0.1,0.1}
\definecolor{brrown}{rgb}{0.6,0.3, 0.0}

\title{ Asymptotics of fundamental solutions for time fractional equations with convolution kernels \thanks{The work was supported by SFB1283 of
German Research Council} }
\author{Yu. Kondratiev$ ^{\scriptscriptstyle 1}$,  A. Piatnitski$^{\scriptscriptstyle 2,3}$ and
E. Zhizhina$^{\scriptscriptstyle 3}$\\[5mm]
$^{\scriptscriptstyle 1}${\footnotesize  Department of Mathematics,  University of Bielefeld, 33501 Bielefeld, Germany.}\\[-0.9mm]
{\footnotesize \tt yukondrat@gmail.com}\\[2mm]
$^{\scriptscriptstyle 2}${\footnotesize  The Arctic University of Norway
Campus Narvik, Postbox 385, }\\[-0.9mm]
{\footnotesize  8505 Narvik, \,
Norway. \ \ \ {\tt apiatni@iitp.ru}}\\[2mm]
$^{\scriptscriptstyle 3}${\footnotesize  Institute for Information Transmission
Problems of RAS,
Moscow 127051,
Russia.} \\[-0.9mm]
{\footnotesize   {\tt ejj@iitp.ru}}
}

\date{}

\begin{document}

\maketitle


\begin{abstract}
The paper deals with the  large time asymptotic of the fundamental solution
for  a time fractional evolution equation for a convolution type operator.
In this equation we use a  Caputo time derivative of order $\alpha$ with
$\alpha\in(0,1)$, and assume that the convolution kernel of the spatial
operator is symmetric, integrable and shows a super-exponential decay at
infinity.  Under these assumptions we describe the point-wise asymptotic behavior of the fundamental
solution in all space-time regions.
\end{abstract}


\section{Introduction and  main results}\label{s_int}

A random time change in Markov processes is motivated by several reasons.
First of all, such change will destroy (in general) the Markov property of the
process. The latter is important in the study of biological models where
the Markov dynamics is a quite  rough approximation to realistic behaviour.
Actually, it is one of possible realizations of a general concept of biological times
specific for such models.

In many areas of theoretical and  experimental physics we meet
a notion of sub-diffusion behavior in stochastic dynamics. In particular,
that is true for dynamics in some composite or fractal media. The random time
techniques give a possibility to realize such sub-diffusion asymptotic
in concrete model situations.

And finally, the random time change in Markov processes is
an interesting and reach source of problems inside of stochastic analysis.

The general framework for a random time change  can be described briefly
as the following scheme. Let $\{X_t,t\geq 0; P_x,x\in E\}$ be a strong Markov process in a phase space $E$.
Denote $T_t$ its transition semigroup (in a proper Banach space) and $L$ the generator of this semigroup.
Let $S_t, t\geq 0$, be a subordinator (i.e., a  non-decreasing real-valued L\'evy process) with $S_0=0$
and the Laplace exponent $\Phi$:
$$
\mathbf Ee^{-\lambda S_t} = e^{-t \Phi(\lambda)} \;\; t,\lambda >0.
$$
We assume that $S_t$ is independent of $X_t$.

Denote $E_t, t>0$, the inverse subordinator and introduce the time changed process
$Y_t = X_{E_t}$.  We are interested in the time evolution
$$
v(x,t) = \mathbf E^x [f(Y_t)]
$$
for a given initial function $f$.   Note that taking informally  $f=\delta$ we arrive at the fundamental
solution of the  related evolution problem.    It is well known,  see e.g. \cite{Toaldo2015}, \cite{Chen2017},
 that $v(t,x)$ is the unique strong solution  to the following  Cauchy problem
$$
\mathbb{D}^{(k)}_t v(x,t) = Lv(x,t) \;\; v(x,0)=f(x).
$$
Here we use a generalized fractional derivative
$$
\mathbb{D}^{(k)}_t \phi (t)= \frac{d}{dt} \int_0^t k(t-s) (\phi(s) -\phi(0))ds
$$
with a kernel $k$ uniquely defined by $\Phi$.

Let $u(x,t)$ be the solution to a similar Cauchy problem but with the ordinary time derivative.
In stochastic terminology,  it is the solution to the  forward Kolmogorov equation corresponding to the
process $X_t$. Under quite general assumptions there is a nice  and essentially obvious relation between these evolutions:
$$
v(x,t)= \int_0^\infty u(x,s) G_t(s)\, ds,
$$
where $G_t(s)$ is the density of $E_t$. Of course, we may have similar relations for fundamental
solutions to considered equations, for the backward Kolmogorov equations or time evolutions of other
related quantities. This technical relation between the random time change and
evolution equations with fractal derivatives is an important technical background
in the study of resulting processes.

Having in mind the analysis of the  influence of the random time change on the asymptotic properties
of $v(x,t)$, we may hope that the latter formula gives all necessary technical equipments. Unfortunately, the situation
is essentially more complicated. The point is about the density $G_t(s)$, in general, our knowledge for a generic
subordinator is very poor. There are two particular cases in which the asymptotic analysis was already realized.
First of all, it is the situation of so-called stable subordinators. Starting with pioneering works by Meerschaert and his collaborators, this
case was studied in details \cite{Baeumer2001,Meerschaert2004}.

Another case is related to a scaling property assumed for $\Phi$ \cite{Chen2018}. It is, nevertheless,
difficult  to
give an interpretation of this scaling assumption in terms of the subordinator.

The problem of asymptotic behaviour of a solution to a fractional evolution equation includes two
essentially different aspects. On the one hand, we should choose certain class of random times.
Another point is a particular type of Markov processes we start with.  In this paper we restrict ourself
to the situation of inverse stable subordinators as random times. Initial Markov processes that we consider are
pure jump homogeneous Markov processes also known as compound Poisson processes or
random walks in $\mathbb{R}^d$ with continuous time.  More precisely, we will be concerned
with the time asymptotic of corresponding fundamental solutions or, that  is the same, related heat
kernels.

Our goal is to describe the large time behavior of the time fractional nonlocal heat kernel $w_\alpha(x,t), \; 0<\alpha<1$, that is a solution of the following fractional time parabolic problem:
\begin{equation}\label{tfhk}
\left\{
\begin{array}{l}
\partial^\alpha_t w_\alpha \ = \ a \ast w_\alpha \ - \ w_\alpha \\
w_\alpha \,|_{t=0} \ = \ \delta_0
\end{array}
\right.
\end{equation}
where $\partial^\alpha_t $ is the fractional derivative (the Caputo derivative of the order $\alpha \in (0,1)$) and
$a(x)$ is a convolution kernel.
We assume that $a(x) \ge 0; \; a(x) = a(-x); \; a(x) \in C_b(\mathbb R^d) \cap L^1(\mathbb R^d)$, and
$$
\int_{\mathbb R^d} a(x) dx =1.
$$
We assume additionally that
the convolution kernel $a(x)$ satisfies for some  $p > 1$ the following condition
\begin{equation}\label{lt}
0 \le a(x) \le C_1 e^{- b |x|^p}.
\end{equation}
Denote by $u(x,t)$ the fundamental solution of  a nonlocal heat equation
\begin{equation}\label{he}
\left\{
\begin{array}{l}
\frac{\partial u}{\partial t} \ = \ a \ast u \ - \ u \\
u|_{t=0} \ = \ \delta_0.
\end{array}
\right.
\end{equation}
Then
\begin{equation}\label{uxt}
u(x,t) \ = \ e^{-t} \delta_0 (x) \ + \ q(x,t)
\end{equation}
with
\begin{equation}\label{v}
q(x,t) \ = \   \sum_{k=1}^{\infty} \frac{ t^k \, e^{-t} }{k!} \  a^{\ast k} (x).
\end{equation}
The function $q(x,t)$ is the regular part of the nonlocal heat kernel $u(x,t)$.

The solution $w_\alpha(x,t)$ of \eqref{tfhk} admits the following representation in terms of the heat kernel $u(x,t)$, see e.g. \cite{Chen2017}, \cite{Chen2018},
\begin{equation*}\label{pxt_bbis}
w_\alpha (x,t)  = \int_0^{\infty} u(x,r) d_r \mathbb{P}(S_r \ge t)  =  \int\limits_0^{\infty} u(x,r) G^{\alpha}_t (r)  dr,
\end{equation*}
where  $S = \{ S_r, \; r \ge 0 \}$ is the $\alpha$-stable subordinator with the Laplace transform $\mathbb{E} e^{-\lambda S_r} = e^{-r \lambda^\alpha}$ and $G^{\alpha}_t(r)= d_r \Pr \{ V^{(\alpha)}_t \le r  \}$ is the density of the inverse $\alpha$-stable subordinator $V^{(\alpha)}_t$.
By relations \eqref{uxt}-\eqref{v} we have
\begin{equation}\label{pxt}
w_\alpha (x,t)
 =  \delta_0 (x) \cdot \int\limits_0^{\infty} G^{\alpha}_t(r)\, e^{-r} dr +  \sum_{k=1}^{\infty} \frac{  a^{\ast k} (x)}{k!} \ \int\limits_0^{\infty} G^{\alpha}_t(r)\, r^k \, e^{-r} dr.
\end{equation}

Using the representation for the Laplace transform of $G^{\alpha}_t(r)$ (see e.g. \cite{UchaZol}):
$${\cal L}(G^{\alpha}_t(r)) = E_\alpha (- \lambda t^\alpha), \quad
E_\alpha \; \mbox{  is the Mittag-Leffler function},
$$
and the properties of the Laplace transform we get for every $k = 0,1,2, \ldots$
$$
\int\limits_0^{\infty} G^{\alpha}_t(r)\, r^k \, e^{-r} dr \ = \ (-1)^k \frac{\partial^k}{\partial \lambda^k} E_\alpha (-\lambda t^\alpha)|_{\lambda=1} \ = \ t^{\alpha k} \, E_\alpha^{(k)}(-t^\alpha).
$$
Consequently representation \eqref{pxt} implies the following formula for $w_\alpha(x,t)$:
\begin{equation}\label{pxtE}
w_\alpha(x,t) \ =   \  E_\alpha (-t^\alpha) \delta_0 (x) \ + \  p_\alpha (x,t),
\end{equation}
where the function $p_\alpha(x,t)$ defined by
\begin{equation}\label{pxtE-bis}
p_\alpha(x,t) \ =  \  \sum_{k=1}^{\infty} \frac{  a^{\ast k} (x)}{k!} \  t^{\alpha k} \, E_\alpha^{(k)}(-t^\alpha)
\end{equation}
is the regular part of $w_\alpha(x,t)$.
Let us notice, that in the case $\alpha=1$ with $E_1(z) = e^z$ we obtain solution \eqref{uxt}, i.e. $w_1(x,t)=u(x,t)$, and $p_1(x,t) = q(x,t)$.

Unfortunately, the elegant formula \eqref{pxtE} could not help much with describing point-wise asymptotics for $p_\alpha(x,t)$, and we choose in this paper an other way of studying the asymptotic behavior of $p_\alpha(x,t)$ which is based on the detailed asymptotic analysis of the function $q(x,t)$ that was done in our previous paper \cite{GKPZ}.

\medskip


Denote by $g_\alpha(s,r), \; s \ge 0$, the density of the $\alpha$-stable subordinator $S_r$. The process $S_r$ has the following self-similarity property:

\medskip

the distribution of $S_r$ is the same as the distribution of $r^{1/\alpha} S_1$.

\medskip
\noindent Consequently
\begin{equation}\label{ss}
g_\alpha(s,r) = r^{-1/\alpha} g_\alpha (s r^{-1/\alpha}), \; s \ge 0,
\end{equation}
where $g_\alpha (s) = g_\alpha (s,1)$ is the density of the $\alpha$-stable law with Laplace transform
$$
\int_0^{\infty} e^{-\lambda s} g_\alpha(s) ds = e^{-\lambda^\alpha}.
$$

In addition, the density $g_\alpha(s), \; s \ge 0$ has the following asymptotics, see e.g. \cite{UchaZol}, \cite{MS}:
\begin{equation}\label{g}
\begin{array}{l}
\displaystyle
g_\alpha(s) \ \sim \ K_\alpha\, \Big( \frac{\alpha}{s} \Big)^{\frac{2-\alpha}{2(1-\alpha)}} \exp \big\{-|1-\alpha|  \Big( \frac{s}{\alpha}\Big)^{\frac{\alpha}{\alpha-1}}   \big\}, \quad \mbox{as } \; s \to 0+; \\[3mm]
\displaystyle
g_\alpha(s) \ \sim \  \frac{\alpha}{\Gamma(1-\alpha)} s^{-\alpha-1},  \quad \mbox{as } \; s \to + \infty,
\end{array}
\end{equation}
with $K_\alpha=\big(2\pi\alpha(1-\alpha)\big)^{-\frac12}$.
Then  the density $ G_t(r)$ of the inverse $\alpha$-stable subordinator $V_t$ has the form
\begin{equation}\label{densityG}
G_t (r) \ = \ \frac{1}{\alpha}\, t \, r^{-1-\frac{1}{\alpha}} \, g_\alpha (t  r^{-\frac{1}{\alpha}}),
\end{equation}
see e.g. \cite{MSK}, \cite{MS}.
The relation \eqref{pxt} implies that the regular part $p = p_\alpha $ of the fundamental solution $w_\alpha$ of the time fractional equation can be written as
\begin{equation}\label{fundamentalsolution}
p(x,t) \ = \   \int_0^{\infty} q(x,r) d_r \mathbb{P}(S_r \ge t) \ = \ \int_0^{\infty} q(x,r) G_t(r) dr.
\end{equation}
In what follows for the sake of brevity we  use the notation $p(\cdot)$ instead of $p_\alpha(\cdot)$.
Using \eqref{densityG} and the change variables $z = t  r^{-{1}/{\alpha}}$ one can rearrange equality \eqref{fundamentalsolution} as
\begin{equation}\label{fundamentalsolution1}
p(x,t) \ = \   \int_0^{\infty} g_\alpha (z) q\big(x, \frac{t^\alpha}{z^\alpha} \big) d z.
\end{equation}
Make in the integral on the right-hand side the change variables
\begin{equation}\label{xi}
s = z^{-\alpha}
\end{equation}
and denote
$$
\hat g_\alpha (s) = g_\alpha (z)|_{z = s^{-1/\alpha}}, \quad W_\alpha (s) = \frac{1}{\alpha} s^{-\frac{1}{\alpha}-1} \hat g_\alpha (s).
$$
Then \eqref{fundamentalsolution1} takes the form:
\begin{equation}\label{fundamentalsolution2}
p(x,t) \ = \   \int_0^{\infty}  \frac{1}{\alpha} s^{-\frac{1}{\alpha}-1} \hat g_\alpha (s) q(x, t^\alpha s) d s \  =  \ \int_0^{\infty} W_\alpha (s) q(x, t^\alpha s) d s.
\end{equation}

\medskip

Notice that in the new variable $s$ defined in \eqref{xi} even for small $s$ such that  $s \gg t^{-\alpha}$ the behaviour of the function $q(x, t^\alpha s)$  is governed by the large time asymptotics of the function $q(x,\tau)$.

Moreover, the asymptotic formulae in \eqref{g} imply the following asymptotics for the function $ W_\alpha (s)$:
\begin{equation}\label{Phi}
\begin{array}{l}
W_\alpha (s) \ \sim \ c_1(\alpha)\, s^{ \frac{1}{2(1-\alpha)} - 1} \exp \{- c_2(\alpha) s^{\frac{1}{1-\alpha}}    \}, \quad \mbox{as } \; s \to \infty; \\ \\
W_\alpha (s)  \ \to \  \frac{\alpha}{\Gamma(1-\alpha)} \frac{1}{\alpha},  \quad \mbox{as } \; s \to 0+,
\end{array}
\end{equation}
with $c_2(\alpha)=(1-\alpha)\alpha\big.^{\frac\alpha{1-\alpha}}$.
Function $ W_\alpha (s) $ coincides with the so-called Wright function, see \cite{GoLuMa}. It readily follows from \eqref{Phi} that the function $ W_\alpha (s) $ has a finite positive limit as $s \to 0+$, and $\int_0^{\infty} W_\alpha (s) d s =1 $ since $W_\alpha$ is a probability density.

\medskip
Representation \eqref{fundamentalsolution2} and the asymptotic formulae in \eqref{Phi} allow one to study the large
time behaviour of $p(x,t)$. It turns out that the asymptotics of $p(t,x)$ depends crucially on the ratio between $|x|$ and $t$.
We consider separately the following regions:
\begin{itemize}
\item $|x|$ is bounded
\item (Subnormal deviations) \ $1\ll|x|\ll t^\frac\alpha2$, or equivalently,  
there exists an increasing function $r(t)$, $r(0)=0$, $\lim\limits_{t \to \infty} r(t) =+ \infty$ such that  $r(t) \le |x| \le (r(t)+1)^{-1} t^{\alpha/2}$  for all sufficiently large $t$.
\item (Normal deviations) \ $x=vt^{\alpha/2}(1+o(1))$, where $v$ is an arbitrary vector in $\mathbb R^d\setminus\{0\}$.
\item (Moderate deviations) \ $x=vt^\beta(1+o(1))$ with $\frac\alpha2<\beta<1$ and $v\in\mathbb R^d\setminus\{0\}$.
\item (Large deviations) \ $x=vt(1+o(1))$ with  $v\in\mathbb R^d\setminus\{0\}$.
\item (Extra large deviations) \ $|x|\gg t$, i.e. $\lim\limits_{t\to\infty}\frac{|x(t)|}{t} = \infty$.
\end{itemize}

\bigskip

\begin{remark}
{\rm
Notice that for any positive function $r(t)$ such that $r(t) \to \infty$ and $ r(t) t^{-\alpha/2} \to 0$, as $t\to\infty$, the set $\{(x,t)\,:\, r(t)<|x|<(1+r(t))^{-1}t^\frac\alpha2$ belongs to the region of subnormal deviations $\{(x,t)\in\mathbb R^d\times(0,+\infty)\,:\, 1\ll|x|\ll t^{\frac\alpha2}\}$}.
\end{remark}
Denote
\begin{equation}\label{Psi}
\Psi(v,s)=\frac1{|\mathrm{det}\sigma|^{1/2}(2\pi s)^{d/2}}\exp\Big(-\frac{(\sigma^{-1} v,v)}s\Big).
\end{equation}

\begin{theorem}
For the function $p(x,t)$ the following asymptotic relations hold as $t \to \infty$: \\

\noindent
1) If $|x|$ is bounded, then

\begin{equation}\label{B1-T}
\begin{array}{ll}
\mathtt{c}_-t^{-\frac\alpha2}\leq p(x,t)\leq\mathtt{c}_+t^{-\frac\alpha2} &\hbox{\rm if }d=1,\\[2mm]
\mathtt{c}_-t^{-\alpha}\log t \leq p(x,t)\leq
\mathtt{c}_+t^{-\alpha}\log t & \hbox{\rm if }d=2,\\[2mm]
\mathtt{c}_-t^{-\alpha}\leq p(x,t)\leq
\mathtt{c}_+t^{-\alpha}  & \hbox{\rm if }d\geq 3.
\end{array}
\end{equation}


\medskip\noindent
2) If  $1\ll|x|\ll t^{\frac\alpha2}$, then
\begin{equation}\label{est_geq3}
\begin{array}{ll}
\mathtt{c}_-t^{-\frac\alpha2}\leq p(x,t)\leq\mathtt{c}_+t^{-\frac\alpha2} &\hbox{\rm if }d=1,\\[2mm]
\mathtt{c}_-t^{-\alpha}\log\Big(\frac{t^\alpha}{|x|^2}\Big)\leq p(x,t)\leq
\mathtt{c}_+t^{-\alpha}\log\Big(\frac{t^\alpha}{|x|^2}\Big)& \hbox{\rm if }d=2,\\[2mm]
\mathtt{c}_-t^{-\alpha}{|x|^{2-d}}\leq p(x,t)\leq
\mathtt{c}_+t^{-\alpha}{|x|^{2-d}} & \hbox{\rm if }d\geq 3.
\end{array}
\end{equation}

\medskip\noindent
3) If  $x=vt^{\alpha/2}(1+o(1))$ with $v \in \mathbb R^d\setminus\{0\}$, then
\begin{equation}\label{B2-T}
p(t^{\alpha/2}v,t)=  t^{-\frac{d\alpha}2}\int\limits_0^\infty W_\alpha(s) \Psi(v,s)\,ds\,\big(1+o(1)\big).
\end{equation}

\medskip\noindent
4) If $x=vt^\beta(1+o(1))$ with $\frac\alpha2<\beta<1$ and $v\in\mathbb R^d\setminus\{0\}$, then
\begin{equation}\label{B3-T}
p(x,t)= \exp\big\{ -K_v \,t^{\frac{2\beta-\alpha}{2-\alpha}} (1+o(1)) \big\}
\end{equation}
with the constant
$$
K_v = (2-\alpha)\alpha\big.^\frac\alpha{2-\alpha} \ \big(\frac12 (\sigma^{-1} v, v) \big)\big.^{\frac1{2-\alpha}}.
$$

\medskip\noindent
5) If $x=vt(1+o(1))$ with  $v\in\mathbb R^d\setminus\{0\}$, then
\begin{equation}\label{B4-T}
p(x,t)=\exp\big\{ -{\mathtt F}(v)t(1+o(1))\big\},
\end{equation}
the function ${\mathtt F}$ is introduced in \eqref{def_f_tt}.

\medskip\noindent
6) If $|x| \gg t$, then
\begin{equation}\label{B5-T}
 p(x,t)\leq
\exp\big\{-\mathtt{c}\big._+|x|\,\big(\log\big|\textstyle{\frac{x}{t}}\big|\big)^{\frac{p-1}p}\big\}
\end{equation}

\end{theorem}

\medskip

\begin{remark}{\rm
Observe that the region of large deviations  $\{(x,t)\,:\,|x|\sim t\}$ for the time fractional heat kernel studied in this work
is the same as that for the heat kernel $q$ of  equation \eqref{he}, \eqref{uxt}.\\
It should also be noted that in the region of extra large deviations $|x|\gg t$ the asymptotic upper bound \eqref{B5-T}
is similar to that obtained in \cite{GKPZ} for $q(x,t)$.
}
\end{remark}

\section{Subnormal deviation region.} \
\label{s_subnormo}
In this section we deal with the region $\{(x,t)\,:\, |x|\ll t^\frac\alpha2\}$.
We consider separately the cases of bounded $|x|$ and growing $|x|$.

\subsection{The case of bounded  $|x|$.} \
\label{ss_bou_x}
In this case
\begin{equation}\label{B1}
\begin{array}{l}
q(x, t^\alpha s) \ \le \ C_1 \min \big\{ t^\alpha s; \; (t^\alpha s)^{-\frac{d}{2}} \big\},\\
q(x, t^\alpha s) \ \ge \ C_2 \min \big\{ t^\alpha s; \; (t^\alpha s)^{-\frac{d}{2}} \big\}
\end{array}
\end{equation}
with some constants $C_1, C_2>0$. Indeed, the estimate by $t^\alpha s$ holds for small value of $\tau = t^\alpha s$, while the estimate $(t^\alpha s)^{-\frac{d}{2}}$ holds for large $\tau = t^\alpha s$.

Using representation \eqref{fundamentalsolution2} we get
\begin{equation}\label{B2}
p(x,t) = \int_0^{\infty} W_\alpha (s) q(x, t^\alpha s) d s \le \tilde C_1 \int_0^{t^{-\alpha}} t^\alpha s d s + C_1 \int_{t^{-\alpha}}^{\infty} W_\alpha (s) (t^\alpha s)^{-\frac{d}{2}} d s.
\end{equation}
The analogous estimate from below holds with an other constant, as follows from \eqref{B1}.
Let us estimate the second integral in \eqref{B2}:
\begin{equation}\label{B2.1}
t^{-\frac{\alpha  d}{2}} \int_{t^{-\alpha}}^{\infty} W_\alpha (s) s^{-\frac{d}{2}} d s.
\end{equation}
Using the properties of the function $W_\alpha (s) $ we get for all $d \neq 2$
\begin{equation}\label{B3}
\int_{t^{-\alpha}}^{1} W_\alpha (s) s^{-\frac{d}{2}} d s + \int_{1}^{\infty} W_\alpha (s) s^{-\frac{d}{2}} d s =
c_3 t^{-\alpha+ \frac{d \alpha}{2}} + c_4,
\end{equation}
and for $d=2$:
\begin{equation}\label{B4}
\int_{t^{-\alpha}}^{1} W_\alpha (s) s^{-\frac{d}{2}} d s + \int_{1}^{\infty} W_\alpha (s) s^{-\frac{d}{2}} d s =
c_5 \alpha \ln t + c_6.
\end{equation}
Here $c_j$ are constants.
Combining \eqref{B2} - \eqref{B4} we obtain the asymptotics \eqref{B1-T} for $p(x,t)$.

\bigskip

\subsection{The case $1\ll|x|\ll t^\frac\alpha2$.} \
\label{ss_almbou_x}

Here we study the asymptotic behaviour of $p(x,t)$ in the region
$\{(x,t)\in\mathbb R^d\times(0,+\infty)\,:\, 1\ll|x|\ll t^{\frac\alpha2}\}$ as $t\to\infty$.

\begin{theorem}\label{t_smallnormal}
Let $r=r(t)$ be an increasing function such that $r(0)=0$ and $\lim\limits_{t\to\infty}r(t)=+\infty$. Then
for all $x\in\mathbb R^d$ such that $r(t)\leq|x|\leq (r(t)+1)^{-1}t^\frac\alpha2$ and for all sufficiently large $t$
we have
\begin{equation}\label{est_geq3}
\begin{array}{ll}
\mathtt{c}_-t^{-\frac\alpha2}\leq p(x,t)\leq\mathtt{c}_+t^{-\frac\alpha2}, &\hbox{\rm if }d=1,\\[2mm]
\mathtt{c}_-t^{-\alpha}\log\big(\frac{t^\alpha}{|x|^2}\big)\leq p(x,t)\leq
\mathtt{c}_+t^{-\alpha}\log\big(\frac{t^\alpha}{|x|^2}\big),& \hbox{\rm if }d=2,\\[2mm]
\mathtt{c}_-t^{-\alpha}{|x|^{2-d}}\leq p(x,t)\leq
\mathtt{c}_+t^{-\alpha}{|x|^{2-d}}, & \hbox{\rm if }d\geq 3.
\end{array}
\end{equation}
\end{theorem}

\begin{proof}
Our arguments rely on the following statement.
\begin{proposition}\label{p_missedbound}
There exist positive constants $c_{j}>0$, $j=1,\,2,\,3,\,4$, such that for all sufficiently large $s>0$ and $x\in\{x\in\mathbb R^d\,:\,|x|\leq s\}$ we have
\begin{equation}\label{f_miest}
c_1{s^{-\frac d2}}\exp\Big(-c_2\frac{|x|^2}s\Big)
\leq q(x,s)\leq c_3{s^{-\frac d2}}\exp\Big(-c_4{\frac{|x|^2}s}\Big)
\end{equation}
\end{proposition}
\begin{proof}[{The proof of this proposition is postponed till Appendix}]
\end{proof}

Let us consider the case $d \ge 3$.
We turn now to the upper bound in \eqref{est_geq3} and consider separately the intervals
$J_1=(0,|x|t^{-\alpha})$,   $J_2=(|x|t^{-\alpha},|x|^\frac32t^{-\alpha})$ and  $J_3=(|x|^\frac32t^{-\alpha},+\infty )$.

By the same arguments as in the proof of Theorem 3.2
in \cite[Section 3.4]{GKPZ} one can derive that
$$
q(x,st^{\alpha})\leq \exp(-c|x|) \qquad \mbox{ for all } \; s\in J_1
$$
with some $c>0$. This implies the inequality
\begin{equation}\label{vklad_j1}
\int_{J_1}W_\alpha(s)q(x,st^\alpha)ds\leq t^{-\alpha}|x|\exp(-c|x|)\leq ct^{-\alpha}|x|^{2-d}.
\end{equation}
According to Proposition \ref{p_missedbound}, for all $s\in J_2$
$$
q(x,st^\alpha)\leq c_3(st^\alpha)^{-\frac d2}\exp\big( -c_4\frac{|x|^2}{st^\alpha}\big)\leq
\exp\big( -c_4|x|^\frac12\big).
$$
Therefore,
\begin{equation}\label{vklad_j2}
\int_{J_2}W_\alpha(s)q(x,st^\alpha)ds\leq t^{-\alpha}|x|^\frac32\exp(-c_4|x|^\frac12)\leq ct^{-\alpha}|x|^{2-d}.
\end{equation}
Using one more time Proposition \ref{p_missedbound}, we obtain
$$
\int\limits_{J_3}W_\alpha(s)q(x,st^\alpha)ds\leq\int\limits_{|x|^\frac32t^{-\alpha}}^\infty
c_3(st^\alpha)^{-\frac d2}\exp\big( -c_4\frac{|x|^2}{st^\alpha}\big)ds
$$
$$
=t^{-\alpha}|x|^{2-d}\int\limits_{|x|^{-\frac12}}^\infty
c_3s^{-\frac d2}\exp\big( -\frac{c_4}{s}\big)ds\leq t^{-\alpha}|x|^{2-d}\int\limits_0^\infty
c_3s^{-\frac d2}\exp\big( -\frac{c_4}{s}\big)ds
$$
Combining the latter estimate with \eqref{vklad_j1} and  \eqref{vklad_j2} yields the desired upper bound in
\eqref{est_geq3}.

In order to  obtain the lower bound in \eqref{est_geq3} we estimate from below the contribution
of the interval $s\in (t^{-\alpha}|x|^2,
2t^{-\alpha}|x|^2)$ as follows
$$
\int\limits_{|x|^2t^{-\alpha}}^{2|x|^2t^{-\alpha}}W_\alpha(s)q(x,st^\alpha)ds\geq
c_5\int\limits_{|x|^2t^{-\alpha}}^{2|x|^2t^{-\alpha}}(st^\alpha)^{-\frac d2}\exp\big( -c_2\frac{|x|^2}{st^\alpha}\big)ds
$$
$$
=c_5t^{-\alpha}|x|^{2-d}\int\limits_1^2s^{-\frac d2}\exp\big( -\frac{c_2}s\big)ds.
$$
This implies the required lower bound.

The cases $d=1$ and $d=2$ can be considered in a similar way.
\end{proof}

\section{Normal deviations region.}
\label{ss_no_dev}

In this section we assume that $x=vt^{\alpha/2}$.

\begin{theorem}\label{t_normal_deva}
Under our standing assumptions on $a(\cdot)$ for any $v \in \mathbb R^d\setminus\{0\}$
we have
\begin{equation}\label{B2-T_second}
p(t^{\alpha/2}v,t)=  t^{-\frac{d\alpha}2}\int\limits_0^\infty W_\alpha(s) \Psi(v,s)\,ds\,\big(1+o(1)\big),
\end{equation}
where $o(1)$ tends to zero as $t\to\infty$.
\end{theorem}
\begin{proof}
In  representation \eqref{fundamentalsolution2} it is convenient
to divide the integration interval into three parts, $J_1=(0, \frac14t^{-\alpha/2})$, $J_2=(\frac14t^{-\alpha/2},\delta )$
and $J_3=(\delta,+\infty)$, where $\delta$ is a sufficiently small number that will be chosen later.

\noindent
{\sl Step 1}. We first estimate the contribution of $J_3$.
 According to \cite[Theorem 19.1]{BhRao} we have
\begin{equation}\label{globest}
\lim\limits_{n\to\infty} \max\limits_{v\in\mathbb R^d}\big| n^{d/2} a^{*n}\big(\sqrt{n}v\big)-\Psi(v,1)\big|=0,
\end{equation}
where the function $\Psi$ was defined in \eqref{Psi}.
This implies in the standard way that for any $\delta>0$
\begin{equation}\label{est_goodarea}
\lim\limits_{t\to\infty}\ \sup\limits_{s\geq\delta, \,v\in\mathbb R^d}\
\big| t^{\frac{d\alpha}2} q\big(t^{\frac\alpha2}v, st^\alpha\big)-\Psi(v,s)\big|=0.
\end{equation}
See the proof of relation \eqref{est_goodarea} in Appendix.
By the Lebesgue theorem
\begin{equation}\label{lim_goodarea}
t^{\frac{d\alpha}2}\int\limits_\delta^\infty W_\alpha(s)q\big(t^{\frac\alpha2}v, st^\alpha\big)\,ds
\longrightarrow \int\limits_\delta^\infty W_\alpha(s) \Psi(v,s)\,ds
\end{equation}
 for each $v\in\mathbb R^d$, as $t\to\infty$.  Consequently,
 \begin{equation}\label{contr_goodarea}
\int\limits_\delta^\infty W_\alpha(s)q\big(t^{\frac\alpha2}v, st^\alpha\big)\,ds
= t^{-\frac{d\alpha}2}\int\limits_\delta^\infty W_\alpha(s) \Psi(v,s)\,ds\,\big(1+o(1)\big),
\end{equation}
where $o(1)$ tends to zero as $t\to\infty$. Observe also that
\begin{equation}\label{hvostik_pred}
\int\limits_0^\delta W_\alpha(s) \Psi(v,s)\,ds\,\to 0,\quad\hbox{as }\delta\to0.
\end{equation}
\medskip

\noindent
{\sl Step 2}.  Next we are going to show that the contribution of the interval $J_2$ is getting negligible
as $\delta\to 0$. To this end we prove that
\begin{equation}\label{step2-1}
q(t^{\alpha/2}v, st^\alpha) \leq  C (v) t^{-\frac{\alpha d}2} \qquad \mbox{for all } \; s \in J_2
\end{equation}
with some constant $C(v)$ that might depend on $v$.  The proof relies on the representation formula for $q(x,t)$ in \eqref{v}.
 In order to extract the terms that provide the main contribution to the sum in the representation of $q(t^{\alpha/2}v, st^\alpha)$ we divide this sum into three parts:
\begin{equation}\label{step2-3}
q(t^{\alpha/2}v, st^\alpha)
= e^{-st^\alpha} \Big\{  \sum\limits_{n=1}^{st^\alpha-(st^\alpha)^{3/4}} + \sum\limits_{n=st^\alpha-(st^\alpha)^{3/4}}^{st^\alpha+(st^\alpha)^{3/4}} + \sum\limits_{n=st^\alpha+(st^\alpha)^{3/4}}^\infty \Big\} \frac{(st^\alpha)^n}{n!} a^{*n}(t^{\alpha/2}v)
\end{equation}
$$
 =  e^{-st^\alpha} \sum\limits_{n = st^\alpha - (st^\alpha)^{3/4}}^{st^\alpha + (st^\alpha)^{3/4}}  \frac{(st^\alpha)^n}{n!} a^{*n}(t^{\alpha/2}v) + O(e^{-c t^{\alpha/4}});
$$
the second relation here is a consequence of the Stirling formula.
Let us estimate $ t^{\frac{d\alpha}{2}} a^{*n}(t^{\alpha/2}v)$ for all  $n \in \big(st^\alpha-(st^\alpha)^{3/4}, st^\alpha + (st^\alpha)^{3/4} \big)$. Notice that $n \to \infty$ as $t \to\infty$ uniformly in $s\in J_3$.
Using \eqref{globest} we have
$$
t^{\frac{d\alpha}{2}}\, a^{*n}(t^{\alpha/2}v) = \frac{(st^\alpha)^{d/2}}{s^{d/2}}\,  a^{*n}\big( (st^{\alpha})^{1/2} \frac{v}{\sqrt{s}} \big)
$$
$$
= \frac{1}{s^{d/2}} n^{d/2}(1+ o(1)) \, a^{*n}\big( \sqrt{n} \frac{v(1+ o(1))}{\sqrt{s}} \big) - \frac{1}{s^{d/2}} \Psi \big(\frac{v(1+ o(1))}{\sqrt{s}}, 1 \big)
$$
$$
+ \frac{1}{s^{d/2}} \Psi \big(\frac{v(1+ o(1))}{\sqrt{s}}, 1 \big)
\to \frac{1}{s^{d/2}} \Psi \big(\frac{v}{\sqrt{s}}, 1 \big) = \Psi(v,s).
$$
Since the function $\Psi(v,s)$ is uniformly bounded for all $s \in (0, \infty)$, then we get
\begin{equation}\label{step2-4}
a^{*n}(t^{\alpha/2}v) \le B(v) t^{-\frac{d\alpha}{2}} \quad \mbox{as } \; t\to \infty
\end{equation}
for all $n \in \big(st^\alpha-(st^\alpha)^{3/4}, st^\alpha + (st^\alpha)^{3/4} \big)$.
Consequently \eqref{step2-3} together with \eqref{step2-4} imply \eqref{step2-1}.

As an immediate consequence of   \eqref{step2-1}  we obtain
\begin{equation}\label{step2-2}
\int\limits_{J_2} q( t^{\alpha/2}v, st^\alpha) W_\alpha(s) \,ds \leq  C_1\delta t^{-\frac{\alpha d}2}.
\end{equation}
This yields the required statement.

\medskip
\noindent
{\sl Step 3}. It remains to estimate the contribution of the interval $J_1$. Again we divide the sum in representation \eqref{v} into two parts:
\begin{equation*}\label{repr_qq}
q(t^{\frac\alpha2}v, st^\alpha)=
e^{-st^\alpha}\sum\limits_{n=1}^{t^{\alpha/2}}\frac{(st^\alpha)^n}{n!} a^{*n}(t^{\frac\alpha2}v)
+e^{-st^\alpha}\sum\limits_{n>t^{\alpha/2}}\frac{(st^\alpha)^n}{n!} a^{*n}(t^{\frac\alpha2}v)=\Sigma_4+\Sigma_5.
\end{equation*}
If $n\geq t^{\alpha/2}$ and $s\leq \frac14 t^{-\alpha/2}$, then after a simple computation we obtain
$$
\exp\big(-st^\alpha\big)\frac{(st^\alpha)^n}{n!}\leq \exp\big(-\kappa_5 t^{\alpha/2}\big)
$$
with some constant $\kappa_5>0$.  Then $\Sigma_5$ admits the following upper bound
\begin{equation}\label{est_sig5}
\Sigma_5 \leq C_5 \exp\big(-\kappa_5 t^{\alpha/2}\big)
\end{equation}
with a positive constant $C_5$.

\medskip

We turn to estimating $\Sigma_4$.
Observe that we sum up over all integer $n$ from the interval $(0,t^\frac\alpha2)$. In particular,
 $n$ need not tend to infinity as $t\to\infty$.

\begin{lemma}\label{Markov}
For any $v\in \mathbb{R}^d \backslash \{0\}$ there exist $c(v)>0$ and $C(v)>0$ such that for all $n <t^{\alpha/2}$ we have
\begin{equation}\label{ME}
a^{*n}(t^{\alpha/2}v) \leq C(v) \exp\big(-c(v) t^{\alpha/2} \big).
\end{equation}
\end{lemma}

\begin{proof}
The proof of the lemma is based on the Markov inequality.
Denote by $S_n$ the sum of $n$ i.i.d. random vectors with a common distribution
density $a(x)$. The distribution density of $S_n$ is $a^{*n}$. The notation $S^j_n$ is used for the $j$-th coordinate of $S_n$.   For $n <t^{\alpha/2}$ and any $r>0$ we have
$$
\int\limits_{|x|>rt^{\alpha/2}}a^{*n}(x)\,dx=\mathbf{P}\{|S_n|\geq rt^{\alpha/2}\}=
\mathbf{P}\big\{|S_n|\geq n\frac{rt^{\alpha/2}}n\big\}
$$
$$
\leq \sum\limits_{j=1}^d
\mathbf{P}\big\{|S^j_n|\geq \frac{n}{d} \, \frac{rt^{\alpha/2}}n\big\}.
$$
According to the Markov's inequality for the terms on the right-hand side of the last estimate the following upper bound holds:
$$
 \mathbf{P}\big\{|S^j_n|\geq n\frac{rt^{\alpha/2}}{dn}\big\}\leq \exp\Big(-\max\limits_{\gamma\in\mathbb R}
\big( \gamma\frac{rt^{\alpha/2}}{dn}-L^j(\gamma)\big)\,n\Big),
$$
where $L^j(\gamma)$ is the cumulant of $S_1^j$.
Under our assumptions on $a(\cdot)$ there is a positive constants $c_0$ such that
$$
L^j(\gamma)\leq c_0\gamma^2
$$
for all $\gamma$ such that $|\gamma|\leq 1$. Since $\frac{t^{\alpha/2}}{dn}>\frac4d$, the latter inequality implies
the following estimate
$$
\max\limits_{\gamma\in\mathbb R}
\big( \gamma\frac{rt^{\alpha/2}}{dn}-L^j(\gamma)\big)\geq  \max\limits_{|\gamma| \le 1}
\big( \gamma\frac{rt^{\alpha/2}}{dn}-L^j(\gamma)\big) \geq c_{d,r}\frac{t^{\alpha/2}}{n}
$$
with a positive constant $c_{d,r}$.
Hence, for any $r>0$ ,
\begin{equation}\label{oi_bouu}
\int\limits_{|x|>rt^{\alpha/2}}a^{*n}(x)\,dx\leq \exp\big(-c_{d,r}t^{\alpha/2}\big).
\end{equation}
Combining this estimate with the estimate $a(x)\leq Me^{-b|x|}$ that is granted by our assumptions on $a$,
and writing $a^{*(n+1)}=a^{*n}\ast a$, one can show in the standard way that
$$
a^{*(n+1)}(t^{\alpha/2}v)\leq C(v) \exp\big(-c(v) t^{\alpha/2}\big).
$$
Indeed,  by \eqref{lt} and \eqref{oi_bouu}
$$
a^{*(n+1)}(t^{\alpha/2}v)=\int\limits_{\mathbb R^d}a^{*n}(y)\,a(t^{\alpha/2}v-y)\,dy.
$$
$$
\leq \int\limits_{|y|\geq \frac12t^{\alpha/2}|v|}a^{*n}(y)\,a(t^{\alpha/2}v-y)\,dy
+\int\limits_{|y|\geq \frac12t^{\alpha/2}|v|}a^{*n}(t^{\alpha/2}v-y)\,a(y)\,dy
$$
$$
\leq \|a\|_{L^\infty}\big(e^{-c(v)t^\frac\alpha2}+ e^{-bc(v)t^\frac{\alpha p}2}\big)
$$
\end{proof}

Inequality \eqref{ME} immediately implies the following upper bound
$$
\Sigma_4\leq \exp(-ct^{\alpha/2})
$$
Combining it with \eqref{est_sig5} yields
\begin{equation}\label{est_contr_j1}
\int\limits_{J_1}q(t^{\alpha/2}v, st^\alpha)\,ds\leq \exp(-ct^{\alpha/2}).
\end{equation}

\medskip\noindent
Finally, from \eqref{contr_goodarea}, \eqref{hvostik_pred},  \eqref{step2-2} and \eqref{est_contr_j1} we deduce that
$$
p(t^{\alpha/2}v,t)=  t^{-\frac{d\alpha}2}\int\limits_0^\infty W_\alpha(s) \Psi(v,s)\,ds\,\big(1+o(1)\big),
$$
where $o(1)$ tends to zero as $t\to\infty$.
\end{proof}

\section{Moderate deviations region.}
\label{ss_mo_dev}

In this section we consider the region $t^{\frac\alpha2}\ll|x|\ll t$. The name "moderate deviations region" is related to the fact that studying the large time behaviour of $p(x,t)$ in this region relies on the asymptotic formulae for $q(x,\cdot)$ in the region
of moderate deviations.   For presentation simplicity we assume
that
\begin{equation}\label{rel_moddev}
  x=v t^{\beta}(1+o(1)) \qquad \hbox{with }\beta\in \big( \textstyle{\frac{\alpha}{2}}, 1\big),
\end{equation}
here $o(1)$ tends to zero as $t\to\infty$.

\begin{theorem}\label{t_umerennost}
Let relation \eqref{rel_moddev} hold with  $\beta \in \big(\textstyle{\frac{\alpha}{2}}, 1\big)$. Then, as $t\to\infty$,
\begin{equation}\label{MD-Kbeta}
p(x,t) = \exp\Big\{-  K_v  t\big.^{\frac{2\beta - \alpha}{2-\alpha}}(1+o(1))\Big\}
\end{equation}
with $K_v = c_3(\alpha) K$,  $c_3(\alpha)=(2-\alpha)\alpha\big.^\frac\alpha{2-\alpha}$, $ K= \big(\frac12 (\sigma^{-1} v, v) \big)\big.^{\frac1{2-\alpha}}$.
\end{theorem}

\begin{proof}
We first prove a lower bound. To this end we let
\begin{equation}\label{sel_xi}
\xi_0=\alpha\big.^{-\frac\alpha{2-\alpha}}\big(\frac12(\sigma^{-1} v, v)\big)\big.^\frac{1-\alpha}{2-\alpha} t\big.^{(2\beta - \alpha)\frac{1-\alpha}{2-\alpha}}.
\end{equation}
According to \cite[Theorem 3.1]{GKPZ}, for all $\xi\in [\xi_0-1,\xi_0+1]$ we have
$$
q(x,t^\alpha\xi)=\exp\Big(- \frac{(\sigma^{-1} x, x)}{2t^\alpha\xi_0}(1+o(1))\Big),
$$
where $o(1)$ tends to zero, as $t\to\infty$, uniformly in $\xi\in [\xi_0-1,\xi_0+1]$.
Combining this relation with \eqref{sel_xi} and the first formula in \eqref{Phi}, after straightforward
computations we obtain
$$
W_\alpha(\xi)q(x,t^\alpha\xi)=
\exp\big\{-c_3(\alpha)\big(\frac12 (\sigma^{-1} v, v) \big)\big.^\frac1{2-\alpha} t\big.^{\frac{2\beta - \alpha}{2-\alpha}} (1+o(1))\big\}
$$
uniformly in $\xi\in [\xi_0-1,\xi_0+1]$.  Integrating the last relation yields the desired
lower bound.

\bigskip

To prove the upper bound for $p(x,t)$ we divide the integration domain into three parts:
$$
J_1 = (0, t^{\beta - \alpha}), \quad J_2 = (t^{\beta - \alpha}, t^{2\beta - \alpha}), \quad J_3 = (t^{2 \beta - \alpha}, \infty),
$$
and  show that the second interval $J_2$ provides the main contribution to the integral in \eqref{fundamentalsolution2}.
We have
\begin{equation}\label{UB-MD-0}
p(x,t) = \int\limits_{J_1} W_\alpha(s) q\big(x, st^\alpha\big)\,ds + \int\limits_{J_2} W_\alpha(s) q\big(x, st^\alpha\big)\,ds + \int\limits_{J_3} W_\alpha(s) q\big(x, st^\alpha\big)\,ds.
\end{equation}

Our first aim is to calculate  the second integral on the right-hand side in \eqref{UB-MD-0}. To this end we split interval $J_2$ into three parts:
$$
J_2^{1} = (t^{\beta - \alpha}, t^{\gamma_1}), \quad J_2^{2} = (t^{\gamma_1}, t^{2\beta - \alpha-\gamma_2}), \quad J_2^{3} = (t^{2 \beta - \alpha-\gamma_2},  t^{2\beta - \alpha}),
$$
if $\beta\leq\alpha$, and
$$
J_2^{1} = (t^{\beta - \alpha}, t^{\beta - \alpha + \gamma_1}), \quad J_2^{2} = (t^{\beta - \alpha + \gamma_1}, t^{2\beta - \alpha-\gamma_2}), \quad J_2^{3} = (t^{2 \beta - \alpha-\gamma_2},  t^{2\beta - \alpha}),
$$
if $\beta>\alpha$.
We then show that for sufficiently small $\gamma_1, \gamma_2 >0$  the contribution of the corresponding integrals over $J_2^{1}$ and $J_2^{3}$  do not exceed
$ o \big( e^{-t^{\frac{2\beta - \alpha}{2-\alpha}}} \big)$ as $t \to \infty$.
Indeed, considering the asymptotics of $W_\alpha(s)$ in  \eqref{Phi} we conclude that on interval  $J_2^{3}$ the following upper bound holds:
$$
W_\alpha(s) \le C_1  t^{m(\alpha,\beta)}   e^{-c_2(\alpha) t^{\frac{2\beta-\alpha - \gamma_2}{1-\alpha}}}, \quad s\in J_2^{3},
$$
with some $m(\alpha,\beta)>0$.   For $0< \gamma_2 < \frac{2\beta-\alpha}{2-\alpha}$ this yields
\begin{equation}\label{UB-MD-1}
\int\limits_{J_2^3} W_\alpha(s) q\big(x, st^\alpha\big)\,ds = o \big( e^{-t^{\frac{2\beta - \alpha}{2-\alpha}}} \big).
\end{equation}

We turn to the interval $J_2^{1}$. If $\frac{\alpha}{2} < \beta \leq\alpha$ then letting
\begin{equation}\label{gamma1-1}
0<\gamma_1 < \frac{(2\beta - \alpha)(1-\alpha)}{2-\alpha}
\end{equation}
we obtain
$$
W_\alpha(s) \le C_2, \quad q\big(r t^\beta, st^\alpha \big) \le e^{-c(r)t^{2\beta-\alpha-\gamma_1}} = o \big( e^{-t^{\frac{2\beta - \alpha}{2-\alpha}}} \big)
$$
for all $s \in J_2^{1} = (t^{\beta - \alpha}, t^{\gamma_1})$.
Analogously, if $\alpha< \beta <1$, then we choose $\gamma_1$ such that
\begin{equation}\label{gamma1-2}
0 < \gamma_1 < \frac{ \alpha (1-\beta)}{2-\alpha}.
\end{equation}
In this case
$$
W_\alpha(s) q\big(v t^\beta, st^\alpha \big) = o \big( e^{-t^{\frac{2\beta - \alpha}{2-\alpha}}} \big),
$$
and consequently
\begin{equation}\label{UB-MD-2}
\int\limits_{J_2^1} W_\alpha(s) q\big(x, st^\alpha\big)\,ds = o \big( e^{-t^{\frac{2\beta - \alpha}{2-\alpha}}} \big).
\end{equation}
It remains to compute the asymptotics of the integral over $J_2^2$.  For $\beta>\alpha$ it takes the form
\begin{equation}\label{UB-MD-3}
\int\limits_{J_2^2} W_\alpha(s) q\big(x, st^\alpha\big)\,ds
= \int\limits_{t^{\beta - \alpha+\gamma_1}}^{t^{2\beta - \alpha-\gamma_2}} W_\alpha(s) q\big(x, st^\alpha\big)\,ds
\end{equation}
The case when  $ J_2^{2} = (t^{\gamma_1}, t^{2\beta - \alpha-\gamma_2})$ ($\beta < \alpha$)  can be considered in a similar way.

Since for all $s \in J_2^2$ we have $s>t^{\beta - \alpha+\gamma_1}$, the function $W_\alpha(s)$ meets the first asymptotics in \eqref{Phi} as $s \in J_2^2$.
Recalling that $x=v t^{\beta}(1+o(1))$,  we represent  $ q\big(x, st^\alpha\big)$ as a sum
\begin{equation}\label{UB-MD-4}
q(v t^{\beta}, s t^\alpha) = e^{-st^\alpha} \left\{ \sum\limits_{k=1}^{t^{(\beta+\gamma_1)}} \ + \ \sum\limits_{k=t^{(\beta+\gamma_1)}+1}^{(1+\delta) s t^\alpha} \ + \ \sum\limits_{k> (1+\delta) s t^\alpha }    \right\} \frac{(st^\alpha)^k}{k!} a^{*k}(v t^{\beta})
\end{equation}
where $\delta>0$ is a sufficiently small positive constant. Notice that the upper summation limit in the second sum on the right-hand side and the lower summation limit in the third sum depend on $s$ that belongs to the interval $ J_2^2 = (t^{\beta - \alpha+\gamma_1}, \ t^{2\beta - \alpha-\gamma_2})$.

We start by estimating the first sum in \eqref{UB-MD-4}.
Using the Markov inequality  in the same was as in the proof of Lemma \ref{Markov} above
we obtain
\begin{equation}\label{UB-MD-5}
\int\limits_{|x|>vt^{\beta}} a^{*k}(x)\,dx  \leq C_d \exp \Big\{ -\max\limits_{\gamma\in\mathbb R}
\big( \gamma\frac{vt^{\beta}}{dk}-L^j(\gamma)\big)\, k \Big\}.
\end{equation}
The maximum on the right-hand side here admits the lower bound
$$
\max\limits_{\gamma\in\mathbb R}
\big( \gamma\frac{v t^{\beta}}{dk}-L^j(\gamma)\big) \ge c_{d,v} \big( \frac{t^\beta}{k} \big)^2 \ge  \frac{t^{\beta- \gamma_1}}{k}
$$
with a constant $c_{d,v}>0$. This yields the following estimate
$$
\int\limits_{|x|>vt^{\beta}} a^{*k}(x)\,dx\leq \exp \big\{ -c_{d,v}t^{\beta - \gamma_1}\big\},
$$
which is valid for any  $k \le t^{\beta+\gamma_1}$.
Combining this estimate with the estimate $a(x)\leq Me^{-b|x|}$ and \eqref{gamma1-2} we conclude that
\begin{equation}\label{UB-MD6}
a^{*(k+1)}( v t^{\beta}) \leq e^{ -c_1 t^{\beta - \gamma_1}} = o \big( e^{-t^{\frac{2\beta - \alpha}{2-\alpha}}} \big) \quad \mbox{for all } \; k \le t^{\beta+\gamma_1}.
\end{equation}
The inequality \eqref{UB-MD6} combined with a trivial inequality
\begin{equation*}\label{sumexp}
e^{-s t^\alpha} \sum\limits_{k=1}^{t^{\beta+\gamma_1}} \frac{(st^\alpha)^k}{k!} \ < \ 1,
\end{equation*}
implies the upper bound for the first sum in \eqref{UB-MD-4}:
\begin{equation}\label{UB-MD-5}
e^{-st^\alpha} \sum\limits_{k=1}^{t^{\beta+\gamma_1}} \frac{(st^\alpha)^k}{k!} a^{*k}(r t^{\beta}) \le C_1 e^{-c_1 t^{\beta - \gamma_1}} = o \big( e^{-t^{\frac{2\beta - \alpha}{2-\alpha}}} \big);
\end{equation}
here we assume that  $\gamma_1$ satisfies \eqref{gamma1-2}.

\medskip
The estimation of the third sum  in  \eqref{UB-MD-4}
is based on the following upper bound
$$
e^{-s t^\alpha} \frac{(st^\alpha)^k}{k!}  \le  e^{- \frac14 \delta^2 s t^{\alpha}} \le e^{-\frac14 \delta^2 t^{\beta+\gamma_1}}  = o \big( e^{-t^{\frac{2\beta - \alpha}{2-\alpha}}} \big),
$$
which is an immediate consequence of the Stirling formula and valid for any $s \ge t^{\beta - \alpha + \gamma_1}$,
$k\geq (1+\delta)st^\alpha$ and  $\delta \in (0, 1)$.
We have also used here an evident inequality $\beta > \frac{2\beta - \alpha}{2-\alpha}$.
Since
$$
\frac{(st^\alpha)^{k+1}\big/(k+1)!}{(st^\alpha)^{k}\big/k!}= \frac{st^\alpha}{(k+1)} <  \frac{st^\alpha}{(1+\delta)st^\alpha} = \frac{1}{1+\delta}
$$
for $k > (1+\delta)st^\alpha$, the third sum on the right-hand side of \eqref{UB-MD-4} can be estimated by the corresponding geometrical progression, and we finally obtain
\begin{equation}\label{UB-MD-5bis}
e^{-st^\alpha} \sum\limits_{k> (1+\delta) s t^\alpha }  \frac{(st^\alpha)^k}{k!} a^{*k}(v t^{\beta}) = o \big( e^{-t^{\frac{2\beta - \alpha}{2-\alpha}}} \big).
\end{equation}

\medskip
The estimation of the second sum in  \eqref{UB-MD-4} with $k \in (t^{\beta+\gamma_1}, \, (1+\delta) s t^\alpha)$ is based on the statement of Lemma 3.14 from \cite{GKPZ}, where the following asymptotic formula for $a^{*k}(x)$ has been justified:
\begin{equation}\label{UB-MD-6}
a^{*k} (x) \ = \ e^{- \frac12 \frac{(\sigma^{-1}x,x)}{k}(1+ \varphi(\frac{x}{k}))}, \quad \mbox{where} \;\; \varphi(\xi) \to 0 \; \mbox{as} \; \xi \to 0,
\end{equation}
provided
\begin{equation}\label{UB-MD-7}
\frac{|x|}{k} \to 0 \quad \mbox{ and } \quad \frac{|x|^2}{k} \to \infty.
\end{equation}
It is easy to see that for all  $k \in (t^{\beta+\gamma_1}, \, (1+\delta) s t^\alpha)$ and $s \in J^2_2$ conditions \eqref{UB-MD-7} are fulfilled:
$$
\frac{|x|}{k} \le C_1 \frac{t^\beta}{t^{\beta+\gamma_1}} \to 0,
\quad
\frac{|x|^2}{k} \ge C_2 t^{2\beta - \alpha -(2\beta-\alpha-\gamma_2)}  \to \infty.
$$
Therefore,  the relation
\begin{equation}\label{UB-MD-8}
a^{*k} (x) \ = \ e^{- \frac12 \frac{(\sigma^{-1}x,x)}{k}(1+ o(1))} \ = \ e^{- \frac12 (\sigma^{-1}v,v) \frac{t^{2\beta - \alpha}}{s}(1+ o(1))}
\end{equation}
holds uniformly for all  $k \in (t^{\beta+\gamma_1}, \, (1+\delta) s t^\alpha)$ as $t\to\infty$.

Combining \eqref{UB-MD-8} with the asymptotic formulae in \eqref{Phi} 
 and taking into account estimates \eqref{UB-MD-5} and \eqref{UB-MD-5bis} for the first and the third sums on the right-hand side of \eqref{UB-MD-4},
we obtain an asymptotic upper bound for the integral in \eqref{UB-MD-3}:
\begin{equation}\label{UB-MD-9}
\int\limits_{J_2^2} W_\alpha(s) q\big(x, st^\alpha\big)\,ds \le e^{-K_v \, t^{\frac{2\beta - \alpha}{2-\alpha}}(1+ o(1))},
\end{equation}
which is valid for all sufficiently large $t$. Here
$$
K_v = c_3(\alpha) K =  (2-\alpha)\alpha\big.^\frac\alpha{2-\alpha}\, \big(\frac12 (\sigma^{-1} v, v) \big)\big.^{\frac1{2-\alpha}}.
$$
It is straightforward to check that
$$
K_v \, t^{\frac{2\beta - \alpha}{2-\alpha}} \ = \ \min\limits_s f(s,t), \qquad
$$
where
$$
f(s,t) = \frac12 (\sigma^{-1}v,v) \frac{t^{2\beta - \alpha}}{s}  + c_2(\alpha)\, s^{\frac{1}{1-\alpha}}, \quad c_2(\alpha)=(1-\alpha)\alpha\big.^\frac\alpha{1-\alpha}.
$$
Notice that  ${\rm argmin} f(\cdot,t) \in J^2_2$.

From \eqref{UB-MD-1},  \eqref{UB-MD-2} and  \eqref{UB-MD-9} one can easily deduce that
\begin{equation}\label{UB-MD-9bis}
\int\limits_{J_2} W_\alpha(s) q\big(x, st^\alpha \big)\,ds \le e^{- K_v \, t^{\frac{2\beta - \alpha}{2-\alpha}}(1+ o(1))}
\end{equation}
with the  constant $ K_v = c_3(\alpha) K$ defined above.

\bigskip
Now we turn to the remaining integrals on the right-hand side  in \eqref{UB-MD-0}. It will be shown  that
\begin{equation}\label{UB-MD-10}
\int\limits_{J_1} W_\alpha(s) q\big(x, st^\alpha\big)\,ds \le O(e^{- c_1 t^\beta}) = o(e^{-t^{\frac{2\beta - \alpha}{2-\alpha}}}),
\end{equation}
and
\begin{equation}\label{UB-MD-11}
\int\limits_{J_3} W_\alpha(s) q\big(x, st^\alpha\big)\,ds \le e^{- c_3 \, t^{\frac{2\beta - \alpha}{1-\alpha}} } = o(e^{-t^{\frac{2\beta - \alpha}{2-\alpha}}}).
\end{equation}
This means in particular that these two integrals do not contribute to the principal term of the  asymptotics of $p(x,t)$.

For $s \ge t^{2\beta-\alpha}$ the asymptotic formula \eqref{Phi} implies that
$$
W_\alpha (s) \le C e^{- c_2(\alpha) t^{\frac{2\beta - \alpha}{1-\alpha}}}.
$$
Since  $q(t^\beta v,,s t^\alpha)$ is bounded for all $t\geq 1$,    we obtain  \eqref{UB-MD-11} with $ c_3 = \frac12 c_2(\alpha) $.

To estimate the integral in \eqref{UB-MD-10} 
we represent  $q \big(v t^{\beta}, st^\alpha \big)$ as a sum
\begin{equation}\label{UB-MD-12}
q(v t^{\beta}, s t^\alpha) = e^{-s t^\alpha} \left\{ \sum\limits_{k=1}^{3 t^{\beta}} \ + \  \sum\limits_{k> 3 t^\beta }    \right\} \frac{(s t^\alpha)^k}{k!} a^{*k}(v t^{\beta})
\end{equation}
For all $k \le 3 t^\beta$ by the Markov inequality in the same way as in the proof of  Lemma \ref{Markov} we have:
\begin{equation*}
\int\limits_{|x|>v t^{\beta}} a^{*k}(x)\,dx  \leq C_d \exp \Big\{ -\max\limits_{\gamma>0}
\big( \gamma\frac{vt^{\beta}}{dk}-L^j(\gamma)\big)\, k \Big\} \leq e^{- c_{d,v} t^\beta}.
\end{equation*}
This yields
\begin{equation*}\label{UB-MD-6}
a^{*(k+1)}( v t^{\beta}) = \int_{\mathbb{R^d}} a^{*k}(vt^\beta - z) a (z) dz \leq \tilde C_1  e^{ -  c_1 t^{\beta}}.
\end{equation*}
The second sum in \eqref{UB-MD-12} can be estimated from above by an appropriate  geometric progression. Indeed, since
for $k > 3 t^\beta$ and $s \le t^{\beta-\alpha}$ we have
$$
\frac{st^\alpha}{(k+1)} <  \frac{t^\beta}{3 t^\beta} = \frac{1}{3},
$$
then the second sum admits the following upper bound:
\begin{equation*}
e^{-st^\alpha} \sum\limits_{k> 3 t^\beta } \frac{(st^\alpha)^k}{k!} a^{*k}(v t^{\beta}) \le \tilde C_2 \frac{(st^\alpha)^{3 t^\beta}}{(3 t^\beta)!} \le  \tilde C_2 e^{-c_2 t^\beta}
\end{equation*}
with $c_2 = 3(\ln 3 -1)$.

\medskip
The relations in \eqref{UB-MD-9} and \eqref{UB-MD-10} - \eqref{UB-MD-11}  yield the desired estimate from above.
\end{proof}

\section{Large deviations region.}
\label{ss_la_dev}

In this section we consider the region of large deviations. Namely, we suppose here that $x=vt(1+o(1))$,
where $v\in \mathbb R^d\setminus\{0\}$.

For the reader convenience we recall here some definitions and statements from \cite{GKPZ}. The notation $I(v)$, $v\in\mathbb R^d$,
is used for the Legendre transform of $L(\cdot)$,  $I(v)=\max\limits_{\gamma\in\mathbb R^d}\big(\gamma\cdot v-
L(\gamma)\big)$.   Under our assumptions on $a(\cdot)$  the function $I$ is smooth and strictly convex in $\mathbb R^d$.
Moreover, $I(0)=0$,  $I(v)>0$ for all $v\in\mathbb R^d\setminus\{0\}$, and
\begin{equation}\label{jurkiy_rost}
\lim\limits_{|v|\to\infty} \frac{I(v)}{|v|}=+\infty.
\end{equation}
The equation
$$
\log\xi=I(\xi v)-\xi v\cdot \nabla I(\xi v), \xi\in\mathbb R^+,
$$
has a unique solution. It is denoted by $\xi_v$.
A function $\Phi(v)$, $v\in\mathbb R^d$, is defined by
$$
\Phi(v)=1-\frac1{\xi_v}\big(1+\log\xi_v-I(\xi_v v)\big).
$$
Then $\Phi$ is a smooth convex function such that $\Phi(0)=0$, $\Phi(v)>0$ if $v\not=0$, and
\begin{equation}\label{jurkiy_rost_snova}
\lim\limits_{|v|\to\infty} \frac{\Phi(v)}{|v|}=+\infty.
\end{equation}

\bigskip
 In order to formulate our results we introduce a function
\begin{equation}\label{defin_F}
F_v(\eta)=c_2(\alpha)\eta\big.^{\frac1{1-\alpha}}+\Phi\Big(\frac v\eta\Big)\eta
\end{equation}
and define
\begin{equation}\label{defin_r_ot_s}
\eta(v)=\mathrm{argmin} F_v(\eta),\qquad \eta\geq 0.
\end{equation}
Since $\Phi(\cdot)$ is a convex function, $F_v(\cdot)$ is a strictly convex function on $(0,+\infty)$.
Due to \eqref{jurkiy_rost_snova} we have
$$
\lim\limits_{\eta\to0} F_v(\eta)=+\infty,\qquad \lim\limits_{\eta\to +\infty} F_v(\eta)=+\infty
$$
for each $v\in \mathbb R^d\setminus\{0\}$.
Consequently, $\eta(v)$ is a well defined function on  $\mathbb R^d\setminus\{0\}$.

Denote
\begin{equation}\label{def_f_tt}
{\mathtt F}(v)=\min\limits_{\eta>0}F_v(\eta)=F_v(\eta(v)).
\end{equation}
\begin{theorem}\label{t_largedev}
Assume that   $x=vt(1+o(1))$ as $t\to\infty$ for some $v\in \mathbb R^d\setminus\{0\}$. Then, as $t\to\infty$,
\begin{equation}\label{asy_la_dev}
p(x,t)=\exp\big(-{\mathtt F}(v)t(1+o(1))\big).
\end{equation}
\end{theorem}

\begin{proof}
We begin by proving the lower bound.
For all $s\in \big({\eta}(v)t^{1-\alpha}-1,{\eta}(v)t^{1-\alpha}+1\big)$
we have
$$
st^\alpha=\eta(v) t(1+o(1)),
$$
where $o(1)$ tends to zero as $t\to\infty$ uniformly in $s\in \big({\eta}(v)t^{1-\alpha}-1,{\eta}(v)t^{1-\alpha}+1\big)$.
According to \cite[Theorem 3.8]{GKPZ} for such $s$ the following relation holds
$$
q(x,st^\alpha)=\exp\Big(-\Phi\big(\frac v{\eta(v)}\big)\eta(v)t(1+o(1))\Big).
$$
Therefore,
$$
W_\alpha(s)q(x,st^\alpha)=\exp\Big(-\big[c_2(\alpha)(\eta(v))^{\frac1{1-\alpha}}
+\Phi\big(\frac v{\eta(v)}\big)\eta(v)\big]t(1+o(1))\Big)\quad\hbox{as }t\to\infty
$$
uniformly in
$s\in\big({\eta}(v)t^{1-\alpha}-1,{\eta}(v)t^{1-\alpha}+1\big)$.
Considering \eqref{fundamentalsolution2} and the definition of ${\mathtt F}$ in \eqref{def_f_tt} we conclude that
$$
p(x,t)\geq\exp\big(-{\mathtt F}(v)t(1+o(1))\big).
$$
This yields the lower bound in \eqref{asy_la_dev}.

\smallskip
We turn to the upper bound. Our first aim is to estimate the contribution of small $s$. According to
\eqref{jurkiy_rost} under our standing assumptions there exists $\gamma_1=\gamma_1(v)>0$ such that
for any $\gamma\leq\gamma_1$ the following inequality holds
$$
I\Big(\frac v\gamma\Big)\gamma >{\mathtt F}(v).
$$
With the help of the Stirling approximation formula, it is straightforward to show that there exists
$\gamma_0=\gamma_0(v)>0$ such that $\gamma_0<\gamma_1$ and for any $\gamma\leq\gamma_0$  we have
$$
\sum\limits_{k\geq \gamma_1 t}\frac{(\gamma t)^k}{k!}e^{-\gamma t}\leq\exp\big(\big[\gamma_1t(\log\gamma-\log\gamma_1)+(\gamma_1-\gamma)t\big]
(1+o(1))  \big)<\exp(-{\mathtt F}(v)t),
$$
where $o(1)$ tends to zero as $t\to\infty$.  Therefore, for any $s\in (0,\gamma_0t^{1-\alpha})$,
$$
\begin{array}{ccc}
q\big(x, st^\alpha\big)<\max\limits_{k\leq \gamma_1t}a^{*k}(x)+C \exp(-{\mathtt F}(v)t)\\[2mm]
\displaystyle
\leq  \max\limits_{k\leq \gamma_1t} \Big\{\exp\Big(-I\Big(\frac{vt}{k}\Big)k\Big)\Big\}+C\exp(-{\mathtt F}(v)t)\ \leq\
\exp\big\{-{\mathtt F}(v)t(1+o(1))\big\}.
\end{array}
$$
   Considering the fact that $W_\alpha$ is a bounded function we
obtain
$$
\int\limits_0^{\gamma_0t^{1-\alpha}} W_\alpha(s) q\big(x, st^\alpha\big)\,ds
\leq  C\int\limits_0^{\gamma_0t^{1-\alpha}} q\big(x, st^\alpha\big)\,ds
\leq \gamma_0t^{1-\alpha}\exp\big\{-{\mathtt F}(v)t(1+o(1))\big\}
$$
\begin{equation}\label{raz}
\leq \exp\big(-{\mathtt F}(v)t(1+o(1))\big).
\end{equation}

\medskip
Due to the first relation in \eqref{Phi} and the fact that $q(x,t)$ is bounded, there exists $\gamma_2>\gamma_1$ such that
\begin{equation}\label{dva}
\int\limits_{\gamma_2t^{1-\alpha}}^\infty W_\alpha(s) q\big(x, st^\alpha\big)\,ds\leq
\exp\big(-{\mathtt F}(v)t\big)
\end{equation}

It remains to estimate the contribution of the interval $s\in(\gamma_0t^{1-\alpha}, \gamma_2t^{1-\alpha})$. Denote
$st^{\alpha}=\gamma t$. Notice that  $\gamma\in(\gamma_0,\gamma_2)$. Then
by \cite[Theorem 3.4]{GKPZ}
$$
q(x,st^{\alpha})\leq \exp\big\{-\Phi\Big(\frac x{st^{\alpha}}\Big)st^{\alpha}(1+o(1))\big\}=
Ct \exp\big\{-\Phi\Big(\frac v\gamma\Big)\gamma t
(1+o(1))\big\},
$$
where $o(1)$ tends to zero as $t\to \infty$ uniformly in $\gamma\in(\gamma_0,\gamma_2)$.
Combining this relation with \eqref{Phi}, \eqref{defin_F}, \eqref{defin_r_ot_s} and \eqref{def_f_tt}, we conclude that
$$
W_\alpha(s)q(x,st^{\alpha})\leq \exp\big\{-{\mathtt F}(v)t(1+o(1))\big\}
$$
and therefore
$$
\int\limits_{\gamma_0t^{1-\alpha}}^{\gamma_2t^{1-\alpha}} W_\alpha(s) q\big(x, st^\alpha\big)\,ds
\leq \exp\big\{-{\mathtt F}(v)t(1+o(1))\big\}.
$$
Estimates \eqref{raz}, \eqref{dva} and the latter relation yield the desired upper bound.
\end{proof}

\medskip

\section{Extra large deviations region.}
\label{ss_sup_la_dev}

In the region of extra large deviation $|x|\gg t$ our asymptotic estimates are not as sharp as in the other regions.
The following statement holds.

\begin{theorem}\label{t_extra_la}
Assume that $|x|\gg t$. Then there exists  a positive constant  ${\mathtt{c}_+}>0$
such that
\begin{equation}\label{sup_la_bou}
 p(x,t)\leq
\exp\big\{-\mathtt{c}\big._+|x|\,\big(\log\big|\textstyle{\frac{x}{t}}\big|\big)^{\frac{p-1}p}\big\}
\end{equation}
for all sufficiently large $t$.
\end{theorem}

\begin{proof}
Our analysis relies again on formula \eqref{fundamentalsolution2}.
We consider separately three intervals: $(0,\infty)=(0,1)\cup \big(1, t^{1-\alpha}\big(\frac{|x|}t\big)^{1-\frac\alpha2}\big) \cup \big(t^{1-\alpha}\big(\frac{|x|}t\big)^{1-\frac\alpha2}, \infty\big)=I_1\cup I_2\cup I_3$.
The fact that the contribution of $s\in(0,1)$ does not exceed the right-hand side of \eqref{sup_la_bou}
is a consequence of Proposition \ref{p_susula} in Appendix. Indeed, since $st^\alpha \ll |x|$ for $s\in (0,1)$, then by
Proposition \ref{p_susula} we obtain
$$
q(x, st^\alpha) \le \exp \big\{  - \mathtt{c}\big._+|x|\,\big(\log\big|\textstyle{\frac{x}{t^\alpha}}\big|\big)^{\frac{p-1}p} \big\}
$$
for all $s \in (0,1)$.

%
For $s\in I_2$ we have $t^\alpha<st^\alpha<\big(\frac{|x|}t\big)^{-\frac\alpha2}|x|\ll |x|$. According to
Proposition \ref{p_susula}, the following estimate holds:
$$
q(x,st^\alpha)\leq \exp\big\{-c_4|x|\,\big[\log\big({\textstyle\frac{|x|}{st^\alpha}}\big)\big]^\frac{p-1}{p}\big\}\leq
\exp\big\{-c_4|x|\,\big[\log\big({\textstyle\frac{|x|}{|x|}\big(\frac{|x|}t\big)^\frac\alpha2}\big)\big]^\frac{p-1}{p}\big\}
$$
$$
\leq \exp\big\{-c_5|x|\,\big[\log\big({\textstyle\frac{|x|}t}\big)\big]^\frac{p-1}{p}\big\}
$$
for some $c_5>0$ and for all sufficiently large $t$ uniformly in $s\in I_2$.
Then
\begin{equation}\label{contri_i1}
\int_{I_2}W_\alpha(s)q(x,t^\alpha s)\,ds\leq \exp\big\{-c_6|x|\,\big[\log\big({\textstyle\frac{|x|}{t}}\big)\big]^\frac{p-1}{p}\big\}.
\end{equation}
In order to estimate the contribution of the interval
$I_3$ we first obtain an upper bound for $W_\alpha(s)$ with $s\in I_3$:
$$
W_\alpha(s)\leq \exp\big\{-c_2(\alpha)t^{-\frac{\alpha}{1-\alpha}}|x|^{\frac1{1-\alpha}}
\big({\textstyle \frac{|x|}t}\big)^{-\frac{\alpha}{2(1-\alpha)}}\big\}
$$
$$
=\exp\big\{-c_2(\alpha)|x|
\big({\textstyle \frac{|x|}t}\big)^{\frac{\alpha}{1-\alpha}-\frac{\alpha}{2(1-\alpha)}}\big\}
\leq \exp\big\{-c_7|x|\,\big[\log\big({\textstyle\frac{|x|}{t}}\big)\big]^\frac{p-1}{p}\big\}.
$$
Therefore,
for sufficiently large $t$ we have
\begin{equation}\label{contri_i2}
\int_{I_3}W_\alpha(s)q(x,t^\alpha s)\,ds\leq \exp\big\{-c_8|x|\,\big[\log\big({\textstyle\frac{|x|}{t}}\big)\big]^\frac{p-1}{p}\big\}.
\end{equation}
To conclude, under a proper choice of a constant $\mathtt{c}_+$ the contribution of each of the intervals  $I_1$, $I_2$ and $I_3$
does not exceed the right-hand side in \eqref{sup_la_bou}. 
This yields \eqref{sup_la_bou}.
\end{proof}

\section*{ Appendix}

Here we prove several inequalities for the fundamental solution $q(x,t)$.

\begin{proof}[Proof of Proposition \ref{p_missedbound}]
We begin with the upper bound.
In the region $\{|x|\leq t^\frac12\log t\}$ we can use  the technique based on the properties of the Fourier transform $\widehat a(\cdot)$
of $a(\cdot)$. We have (see, for instance, formula (2.6) in \cite{GKPZ})
$$
q(x,t)=\int_{\mathbb R^d} e^{ix\cdot p}\big(e^{-t(1-\widehat a(p))}-e^t\big)\,dp.
$$
From this formula, considering our assumptions on $a(\cdot)$, one can easily derive the desired upper bound. We leave the details to the reader.

If $|x|\geq t^\frac12\log t$ then for any $\delta>0$ and all sufficiently large $t$ we have
$$
\exp\big\{-\delta\frac{|x|^2}2\big\}\leq t^{-\frac d2}.
$$
It was shown in the proof of \cite[Lemma 3.18]{GKPZ} that for all $k\geq |x|$ the following inequality holds:
$$
a^{*k}(x)\leq \exp\big\{ -I(\frac xk)k\big\}\leq  \exp\big\{ -\mathtt{c}\frac{|x|^2}k\big\}.
$$
Combining this inequality with the Stirling formula we conclude that for some constant $c>0$ and for all sufficiently large $t$
the following estimate holds:
$$
q(x,t)\leq \exp\big\{ -c\frac{|x|^2}t \big\}\leq t^{-\frac d2} \exp\big\{ -(c-\delta)\frac{|x|^2}t \big\}.
$$
This yields the desired upper bound.

\end{proof}

\begin{proposition}\label{p_susula}
Under our standing assumptions on $a(\cdot)$ there exists a constant $\mathtt{c}>0$ such that  in the region $\{(x,t)\,:\, t>0,\ \frac{|x|}t\gg 1\}$ the following upper bound holds:
\begin{equation}\label{ub_sula}
 q(x,t)\leq
\exp\big\{-\mathtt{c}|x|\,\big(\log\big|\textstyle{\frac{x}{t}}\big|\big)^{\frac{p-1}p}\big\}
\end{equation}
\end{proposition}
\begin{proof}
 We use representation \eqref{v}.
 According to estimate (3.61) in \cite{GKPZ},  there exist  constants $\alpha_p>0$ and $\varkappa>0$ such that
 for all sufficiently large $x$ and for all $k$ with $1\leq k\leq\alpha_p\,|x|$ we have
 $$
 a^{*k}(x)\leq \exp\big\{-\varkappa\frac{|x|^p}{k^{p-1}}\big\}.
 $$
 If $k$ satisfies the estimate $1\leq k\leq |x| \big(\log\big(\frac{|x|}{t}\big)\big)^{-\frac1p}$, then
 $$
 a^{*k}(x)\leq\exp\big\{- \varkappa |x|  \big[\big(\log\big(\frac{|x|}{t}\big)\big)^{\frac1p}\big]^{p-1}\big\}
 =\exp\big\{- \varkappa |x|  \big(\log\big(\frac{|x|}{t}\big)\big)^{\frac {p-1}p}\big\}.
 $$
 We also have
 $$
 \sum_{k\leq |x| \big(\log\big(\frac{|x|}{t}\big)\big)^{-\frac1p}}\frac{ t^k \, e^{-t} }{k!} \ 
 \leq 1.
 $$
 Notice that the relation $|x|\gg t$ implies $|x| \big(\log\big(\frac{|x|}{t}\big)\big)^{-\frac1p}\gg t$.
If  $k\geq |x| \big(\log\big(\frac{|x|}{t}\big)\big)^{-\frac1p}$, then, by the Stirling formula,
$$
\frac{t^k}{k!}\leq \exp\big\{-k\log\big(\frac kt\big)+k\big\}\leq\exp\big\{-\frac12|x|\big(\log\big(\frac{|x|}t\big)\big)^{-\frac1p}
\log\big(\frac{|x|}t\big)\big\}
$$
$$
\leq\exp\big\{- \frac12 |x|  \big(\log\big(\frac{|x|}{t}\big)\big)^{\frac {p-1}p}\big\}.
$$
Combining the last three estimates yields the desired inequality in  \eqref{ub_sula}.
\end{proof}
\begin{remark}{\rm
It should be noted that in the formulation of Proposition \ref{p_susula} the value of $t$
might be arbitrarily small. The only relation that matters is $\frac{|x|}{t}\gg 1$.}
\end{remark}

\medskip
Next we prove \eqref{est_goodarea}.
\begin{proposition}\label{est_gooda}
For any $\delta>0$
\begin{equation*}\label{est_goodareabis}
\lim\limits_{t\to\infty}\ \sup\limits_{s\geq\delta, \,v\in\mathbb R^d}\
\big| t^{\frac{d\alpha}2} q\big(t^{\frac\alpha2}v, st^\alpha\big)-\Psi(v,s)\big|=0.
\end{equation*}
\end{proposition}

\begin{proof}
We divide the sum in formula \eqref{v} into three parts as follows:
$$
q(x,t)=e^{-t} \sum\limits_{n=1}^\infty \frac{t^n}{n!} a^{*n}(x)
= e^{-t} \Big\{  \sum\limits_{n=1}^{t-t^{3/4}} + \sum\limits_{n=t-t^{3/4}}^{t+t^{3/4}} + \sum\limits_{n=t+t^{3/4}}^\infty \Big\} \frac{t^n}{n!} a^{*n}(x).
$$
With the help of the Stirling formula one can easily check that the first and the last sums here are of order
$O\big(e^{-c \sqrt{t}}\big)$ as $t\to\infty$. Therefore,
\begin{equation}\label{App0}
q(x,t)= e^{-t} \sum\limits_{n=t-t^{3/4}}^{t+t^{3/4}}  \frac{t^n}{n!} a^{*n}(x) + O(e^{-c \sqrt{t}}).
\end{equation}
We need to estimate the quantity $t^{\frac{d \alpha}{2}}  a^{*n}(t^{\frac\alpha2}v)$ with  $n \in \big( s t^\alpha - (s t^\alpha)^{\frac34}, s t^\alpha + (s t^\alpha)^{\frac34} \big) $. Observe that for the function $\Psi(v,s)$ defined by \eqref{Psi}  the following relation holds:
\begin{equation}\label{App1}
\Psi(v,s) = \frac{1}{s^{d/2}} \Psi(\frac{v}{\sqrt{s}},1).
\end{equation}
Then from the uniform in $v$ estimate \eqref{globest} 
we deduce
$$
\frac{(st^\alpha)^{d/2}}{s^{d/2}} a^{*n}\big( (s t^{\alpha})^{1/2} \frac{v}{\sqrt{s}} \big)
= \frac{1}{s^{d/2}} n^{d/2}(1+ o(1)) a^{*n}\big( \sqrt{n} \frac{v}{\sqrt{s}}(1+ o(1)) \big)
 \to
$$
$$
 \to \frac{1}{s^{d/2}} \Psi(\frac{v}{\sqrt{s}},1) = \Psi(v,s);
$$
 here the inequality  $n\geq \frac\delta2 t^\alpha$ has been used.
Thus, for any $v \in \mathbb{R}^d$,
\begin{equation}\label{App2}
\max\limits_{n \in ( s t^\alpha - (s t^\alpha)^{\frac34}, s t^\alpha + (s t^\alpha)^{\frac34} )}\ \   \Big|t^{\frac{d \alpha}{2}}  a^{*n}(t^{\alpha/2} v) - \Psi(v,s)\Big| \to 0, \qquad \mbox{as } \; t \to \infty.
\end{equation}
Moreover, the convergence is uniform with respect to  $s \ge \delta$.
Finally from \eqref{App0} and \eqref{App2} we obtain
$$
t^{\frac{d \alpha}{2}} q(t^{\alpha/2} v, st^\alpha) - \Psi(v,s) = e^{- st^\alpha}\sum\limits_{n=1}^\infty \frac{(st^\alpha)^n}{n!} a^{*n}(t^{\alpha/2}v) \, t^{\frac{d \alpha}{2}} - e^{- st^\alpha}\sum\limits_{n=1}^\infty \frac{(st^\alpha)^n}{n!} \, \Psi(v,s)
$$
\begin{equation}\label{App3}
= e^{-st^\alpha}\sum\limits_{n=st^\alpha-(st^\alpha)^{3/4}}^{st^\alpha + (st^\alpha)^{3/4}} \frac{(st^\alpha)^n}{n!} \Big[ a^{*n}(t^{\alpha/2}v) \, t^{\frac{d \alpha}{2}} - \Psi(v,s) \Big] \ + \ O(e^{-c \sqrt{t}}) \to 0.
\end{equation}
This yields \eqref{est_goodarea}.
\end{proof}

\end{document}